\documentclass[12pt]{amsart}
\usepackage{graphicx,  amssymb,amsthm, amsfonts, latexsym, epsfig, amsmath}

\newtheorem{definition}{Definition}[section]
\newtheorem{theorem}[definition]{Theorem}
\newtheorem{lemma}[definition]{Lemma}
\newtheorem{proposition}[definition]{Proposition}

\newtheorem{claim}[definition]{Claim}

\newcommand{\invlim}[2]{\ensuremath{\lim\limits_{\leftarrow}\{#1,#2\}}}

\newcommand{\B}[1]{\ensuremath{\mathbb{#1}}}
\newcommand{\N}{\B{N}}
\newcommand{\Z}{\B{Z}}

\newcommand{\C}[1]{\mathcal{#1}}

\newcommand{\ov}[1]{\ensuremath{\overline{#1}}}

\newcommand{\ilim}[1]{\ensuremath{\lim\limits_{\leftarrow}\{[0,1],
#1\}}}

\newcommand{\n}{\noindent}

\newcommand{\al}{\alpha}

\newcommand{\nin}{\not\in}

\renewcommand{\include}{\input}

\newcommand{\lt}{\operatorname{lt}}

\newcommand{\w}{\omega}

\newcommand{\alp}{\alpha}

\newcommand{\bt}{\beta}

\newcommand{\gm}{\gamma}

\newcommand{\nowt}{\emptyset}
\renewcommand{\include}{\input}
\newcommand{\hi}{\operatorname{Ht}}
\newcommand{\sh}{\sigma}

\newcommand{\aux}{\vartriangleleft}
\renewcommand{\hi}{\operatorname{ht}_{\C T}}
\newcommand{\Hi}{\operatorname{Ht}}
\newcommand{\sartre}{(\,)}

\newcommand{\eps}{\epsilon}

\newcommand{\Fi}{\operatorname{Fi}}
\newcommand{\Q}{\mathbb{Q}}
\newcommand{\lmb}{\lambda}

\begin{document}

\title[Inverse limits of $\w$-limit sets]
{Countable inverse limits of postcritical $\w$-limit sets fo
unimodal maps}

\author[C. Good]{Chris Good}
\address{School of Mathematics and Statistics, University of Birmingham,
Birmingham, B15 2TT, UK} \email{c.good@bham.ac.uk}

\author[R. Knight]{Robin Knight}
\address{Mathematical Institute, University of Oxford,
Oxford, OX1 3LB, UK} \email{knight@maths.ox.ac.uk}

\author[B.E. Raines]{Brian Raines}\thanks{Raines supported by NSF grant 0604958.}
\address{Department of Mathematics, Baylor University, Waco, TX
76798--7328,USA} \email{brian\_raines@baylor.edu}

\maketitle
\begin{abstract}
  Let $f$ be a unimodal map of the interval with critical point $c$.
   If the orbit of $c$ is not dense then most points in $\ilim{f}$
   have neighborhoods that are homeomorphic with the product of a
   Cantor set and an open arc.  The points without this property are
   called \emph{inhomogeneities}, and the set, $\C I$, of
   inhomogeneities is equal to $\invlim{\w(c)}{f|_{\w(c)}}$.  In this
   paper we consider the relationship between the limit complexity
   of $\w(c)$ and the limit complexity of $\C I$.  We show that if
   $\w(c)$ is more complicated than a finite collection of
   convergent sequences then $\C I$ can have arbitrarily high limit
   complexity.  We give a complete description of the limit
   complexity of $\C I$ for any possible $\w(c)$.
\end{abstract}

\section{Introduction}  Let $f$ be a unimodal map of the interval
with critical point $c$.  The structure of the inverse limit space
generated by $f$ has been the subject of much study,
\cite{bargebrucks&diamond}, \cite{barge&diamond}, \cite{barge&martin}.  Recently these spaces have appeared in
the study of certain dynamic models from macroeconomics,
\cite{mediorainesecon}, \cite{mediorainesmath},
\cite{kennedystockmanyorke}, \cite{kennedystockman}.
 A driving problem in the study of such spaces is Ingram's
Conjecture. It states that if $f$ and $g$ are two tent maps with
$f\neq g$, then $\ilim{f}$ is not homeomorphic to $\ilim{g}$.
Ingram's conjecture has been proven in the case that $c$ is
periodic, \cite{kailhofer},
\cite{blockjagkeeskail}, or preperiodic, \cite{bruin},
\cite{stimac:preperiod}.  It has recently been proven in the case that $c$ is
non-recurrent, \cite{raines-stimac}.  One of the distinguishing
features among these various cases is the topological structure of
the postcritical omega-limit set, $\w(c)$.

If the orbit of $c$ is not dense in $[0,1]$ then most points in
$\ilim{f}$ have a neighborhood homeomorphic to the product of a
Cantor set and an open arc.  The exceptional points without this
property we call \emph{inhomogeneities} and denote the collection of
such points by $\C I$.  It is known that $\C
I=\invlim{\w(c)}{f|_{\w(c)}}$, \cite{raines:nonhyp}, and hence it is
a compact subset of $\ilim{f}$.  If $c$ has a dense orbit then
$\w(c)=[0,1]$ and so $\C I=\ilim{f}$.  But it can also be the case
that $\C I$ is finite, countably infinite, or uncountable (in which
case it contains a Cantor set.)

In this paper we consider the case that $\C I$ is countably
infinite.  Then $\C I$ and $\w(c)$ are countably infinite compact
sets, and such sets can be topologically identified by their
\emph{limit complexity} or \emph{limit height}.  Roughly speaking,
we say that a single isolated point has limit height zero.  Then a
point which is a limit of a sequence of isolated points (such as $0$
in the set $\{0\}\cup\{1/n\}_{n\in \N}$) is said to have limit type
$1$, and we say that a point which is a limit of points of type $1$
is of limit type $2$.  To see an example of a point of limit type
$2$, for each $n\in \N$ let $S_n=\{1/n\}\cup \{1/n-1/j\}_{j>n^2}$
and $S=\{0\}\cup \bigcup_{n\in \N}S_n$.  Then $0\in S$ has limit
height $2$.  We can inductively define limit height $n$ for any
$n\in \N$. We can extend the notion to all countable ordinals,
$\gamma$, by defining a point to have limit height $\gamma$ provided
every neighborhood of the point contains points with limit type
$\lambda$ for all $\lambda<\gamma$.

In \cite{goodknightraines}, we prove that the limit height of a
countably infinite $\w$-limit set cannot be a limit ordinal.  Then
for each allowed countable limit height, $\alpha+1$, we construct a
tent map with postcritical $\w$-limit set having limit type
$\alpha+1$ and $\C I$ also having limit type $\alpha+1$.  As a
result of our particular construction techniques in that paper, the
limit type of $\C I$ and $\w(c)$ is identical in our examples.

The question we address in this paper is: Given a countable
postcritical omega-limit set with $n$ points of highest limit type,
$\alpha+1$, that has a countable inverse limit $\C I$ with $m$
points of highest limit type $\beta+1$, what are the relationships
between $\alp$ and $\beta$ and between $n$ and $m$. We completely
answer this question, thus providing a complete picture of the
correspondence  between $\w(c)$ and $\C
I=\invlim{\w(c)}{f|_{\w(c)}}$ for l.e.o. unimodal maps when $\C I$
is countable. In particular, we show that if $\alpha$ is $0$ then so
is $\beta$, but if $\alpha\ge 1$ then $\beta$ can be any countable
ordinal greater than or equal to $\alpha$.

\medskip

In Section \ref{prelims} we briefly mention the symbolic dynamics
and kneading theory that we make use of in this paper. Section
\ref{count} formally defines what we mean by the \emph{limit type}
of a point and describes the structure of countable compact metric
spaces. In this section, we also discuss the notion of a
\emph{well-founded tree} from descriptive set-theory that we use to
construct our kneading sequences. Our main theorem is then stated in
Section \ref{main}, whilst the proof is given in Section
\ref{constructions}, where we prove the existence of various
appropriate unimodal functions by via their kneading sequences, and
in Section \label{rectrictions}, where we discuss the various
restrictions on the correspondence between $\w(c)$ and $\C I$. We
conclude in the final section, by showing that $\C I$ can be
uncountable, even though $\w(c)$ is countable.

\section{Preliminaries}\label{prelims}

We encourage the reader unfamiliar with techniques from the theory
of inverse limit spaces to see \cite{ingram5} or
\cite{ingram&mahavier}.

Let $f:[0,1]\to[0,1]$ be continuous.  We define the \emph{inverse
limit} of $f$ to be the space $$\ilim{f}:=\{x=(x_0, x_1,\dots)\in
[0,1]^\N:f(x_i)=x_{i-1}\}$$  For each $n\ge 0$ we define
$\pi_n:\ilim{f}\to [0,1]$ to be the $n$th projection map, i.e.
$\pi_n(x_0,x_1,\dots)=x_n$.  The standard metric on $\ilim{f}$ we
denote by $d$ and it is defined as
$$d[x,y]=\sum_{i=0}^\infty\frac{|x_i-y_i|}{2^{i+1}}$$

%

We say a map, $f$, is \emph{locally eventually onto, l.e.o.,}
provided for every $\eps>0$ and for all $x\in [0,1]$ there is an
integer $N$ so that $f^N[B_\eps(x)]=[0,1]$. Let $h$ be a unimodal
map of the unit interval with critical point $c_h$.  We have shown
in \cite{rainesthesis} \& \cite{raines:nonhyp} that if $h$ is l.e.o.
and $x\in \ilim h$ such that $x$ does not have a neighborhood
homeomorphic to the product of a Cantor set and an open arc, then
every coordinate of $x$ is an element of the \emph{omega-limit set}
of $c_h$, $\w(c_h)=\bigcap_{n\in \N}\ov{\{h^m(c_h):m\ge n\}}$.  Let
$\C I$ denote the set of such points i.e. $\C
I=\invlim{\w(c)}{f|_{\w(c)}}$.

\begin{lemma} \label{omega cantor implies fold cantor}
 Let $h$ be the core of a tent map with critical point $c_h$ and
$\w(c_h)$ a Cantor set.  Then $\C I$ is also a Cantor set.
\end{lemma}

\begin{proof}  A point $ x$ is in $\C I$
if and only if $\pi_n(x)$ is an element of $\w(c_h)$ for every $n\in
\N$.  Let $x\in \C I$.  Let $\eps>0$ and choose $\delta>0$ and $N\in
\N$ so that if $ z\in \ilim h$ and $\pi_N(z)$ is within $\delta$ of
$\pi_N(x)$ then $d[z,  x]<\eps$. Since $\w(c_h)$ is a Cantor set
there is a point, $z_N\in \w(c_h)$ such that $|z_N-\pi_N(
x)|<\delta$. Since $h[\w(c_h)]=\w(c_h)$ it is easy to see that we
can construct a point $ z\in \C I$ with $\pi_N( z)=z_N$. Hence $ x$
is not isolated in $\C I$ and $\C I$ is a Cantor set.
\end{proof}

Let $i_h:[0,1]\to \{0,1,*\}$ be defined by

$$i_h(x)\begin{cases} 0&\text{if $x<c_h$}\\[3pt]
1&\text{if $x>c_h$}\\[3pt]
*&\text{if $x=c_h$.}\\
\end{cases}$$

\n For each point $x\in [0,1]$ define the \emph{itinerary} of $x$ by

$$I_h(x)=(i_h(x), i_h\circ h(x), i_h\circ h^2(x),\dots)$$

\n and, given an integer $M$, let the \emph{cylinder of diameter $M$
centered on $x$} be given by:

$$I_h(x)|_{M}=(i_h(x), i_h\circ h(x), i_h\circ h^2(x), \dots,
i_h\circ h^M(x)).$$

\n The \emph{kneading sequence} for $h$ is defined to be
$K_h=I_h[h(c_h)]$.

The following results are well-known and easy to prove,
\cite{colletandeckmann}.  Let $P_h\subseteq [0,1]$ be the collection
of precritical points for $h$, i.e. the collection of points that
have a $*$ in their itinerary.

\begin{lemma}  \label{epsfirst} Let $h:[0,1]\to [0,1]$ be unimodal and l.e.o. and $\eps>0$.
Then there is an integer $N$ such that if $x,y\in [0,1]\setminus
P_h$ and $I_h(x)|_{N}=I_h(y)|_N$  then $|x-y|<\eps$.
\end{lemma}

\begin{lemma} \label{intfirst}  Let $h:[0,1]\to [0,1]$ be unimodal and l.e.o. and choose $N\in
\N$.  Then there is an $\eps>0$ such that if $x,y\in [0,1]\setminus
P_h$ and $|x-y|<\eps$ then $I_h(x)|_{N}=I_h(y)|_{N}$.
\end{lemma}

\begin{theorem} \label{omegaident}  Let $x,y\in [0,1]$.  Then
$x\in \w(y)$ if and only if, for every $N\in \N$, $I_h(x)|_{N}$
occurs infinitely often in $I_h(y)$.
\end{theorem}

\n For each point $x=(x_0, x_1, x_2\dots)$ in $\ilim h$ define the
\emph{full itinerary} for $x$ by

$$\Fi_h(x)=\left(\dots, i_h(x_3), i_h(x_2), i_h(x_1)\; .\; i_h(x_0), i_h\circ
h(x_0), i_h\circ h^2(x_0),\dots\right)$$

\n Notice that if $h$ is l.e.o. then $\Fi_h$ is a one-to-one map.
Given a bi-infinite sequence $Z=(\dots \zeta_{-2}, \zeta_{-1}\; . \;
\zeta_{0}, \zeta_1, \zeta_2, \dots)$ define the \emph{shift map} by
$\hat\sh(Z)=(\dots  \zeta'_{-2}, \zeta'_{-1}\; . \; \zeta'_{0},
\zeta'_1, \zeta'_2,\dots )$ where $\zeta_i'=\zeta_{i-1}$. Define the
\emph{backwards itinerary} for $x$ by

$$\Fi_h^-(x)=(\dots, i_h(x_3), i_h(x_2), i_h(x_1))$$


Given a point $x\in \ilim h$ and an integer $M$ call the string

$$\Fi_h(x)|_{-M,M}=(i_h(x_M),\dots , i_h(x_1)\; .\; i_h(x_0),
i_h\circ h(x_0),\dots , i_h\circ h^M(x_0))$$

\n the \emph{cylinder} of diameter $M$ of $\Fi_h(x)$.

The proofs of the following lemmas are straightforward and can be
found in \cite[Lemmas 2.5 \& 2.6]{goodraines}.

\begin{lemma}  Let $h:[0,1]\to[0,1]$ be l.e.o. and unimodal.  Let $\eps>0$.
Then there is a positive integer $M$ with the property that if $ x,
 y\in \ilim h$ with $\Fi_h(x)|_{-M,M}=\Fi_h(
y)|_{-M,M}$ and neither of $\Fi_h(x)$ and $\Fi_h( y)$ contain $*$
then $d[x,  y]<\eps$.
\end{lemma}


\begin{lemma}  Let $h:[0,1]\to[0,1]$ be l.e.o. and unimodal.
Let $M\in \N$.  Then there is an $\eps>0$ so that if $x, y\in \ilim
h$ with $d[x, y]<\eps$ and neither $\Fi_h(x)$ nor $\Fi_h( y)$
contain $*$, then $\Fi_h(x)|_{-M, M}=\Fi_h(y)|_{-M, M}$.
\end{lemma}
%


A sequence, $M$, in symbols $0$ and $1$ is {\em primary} provided it
is not a $*$-product, i.e. there is no finite word $W$ and sequence
$(u_i)_{i\in \N}$ of points from $\{0,1\}$ with
$M=Wu_1Wu_2Wu_3\dots$. The {\em shift map}, $\sh$, on sequences is
defined by $\sh[(t_0,t_1,\dots)]=(t_1,t_2,\dots)$.  We order
sequences using the {\em parity-lexicographic ordering}, $\prec$. To
define this order we first define $0<*<1$.  Let
$t=(t_0,t_1,t_2,\dots)$ and $s=(s_0,s_1,s_2,\dots)$ be sequences of
zeroes and ones.  Let $n$ be the least $j$ such that $t_j\neq s_j$.
Let $m$ be the number of occurrences of the symbol $1$ in the string
$(t_0,t_1,\dots, t_{n-1})=(s_0,s_1,\dots, s_{n-1})$. If $m$ is even
then define $t\prec s$ if, and only if, $t_n<s_n$. If $m$ is odd
then define $t\prec s$ if, and only if $t_n>s_n$. It is easy to show
that if $x<y$ then $I_f(x)\prec I_f(y)$.  A sequence, $K$, is {\em
shift-maximal} provided that for all $j\in \N$, $\sh^j(K)\prec K$ or
$\sh^j(K)=K$.

The following theorem allow us to construct an infinite sequence of
$0$'s and $1$'s that is the kneading sequence for a tent map core.

\begin{theorem}\cite[Lemma
III.1.6]{colletandeckmann}\label{tentpar} Let $K$ be a infinite
sequence of $0$s and $1$s that is shift-maximal, primary and
$101^\infty\preceq K$. Then there is a parameter, $q$, in
$[\sqrt{2},2]$ generating a tent map core, $f_q$, with kneading
sequence $K$.
\end{theorem}

\section{Countable Compact Metric Spaces and Limit
Types}\label{count}

In this section we briefly describe the structure of compact,
countable metric spaces and define what we mean by the limit type or
limit complexity. We then define the notion of a well-founded tree,
which we use in the construction of our kneading sequences in
Section \ref{constructions}.

Let $X$ be any topological space and let $A$ be a subset of $X$. The
\emph{Cantor-Bendixson derivative}, $A'$ of $A$, is the set of all
limit points of $A$. Inductively, we can define the \emph{iterated
Cantor-Bendixson derivatives} of $X$ by
\begin{align*}
X^{(0)} & = X,\\
X^{(\alpha+1)}& = \left(X^{(\al)}\right)',\\
X^{(\lambda)}& = \bigcap\limits_{\al<\lambda}X^{(\al)} \mbox{ if
$\lambda$ is a limit ordinal.}
\end{align*}
Clearly for some ordinal $\gm$, $X^{(\gm)}=X^{(\gm+1)}$ and $X$ is
said to be \emph{scattered} if this set is empty and $X$ is
nonempty. In this case, a point of $X$ has a well-defined
Cantor-Bendixson rank, often called the \emph{scattered height} or
\emph{limit type} of $x$, defined by $\lt(x)=\alp$ if and only if
$x\in X^{(\alp)}\setminus X^{(\alp+1)}$. The $\alp^{\text{th}}$
level $L_\alp$ of $X$ (or, more formally, $L^X_\alp$) is then the
set of all points of limit type $\alp$. Clearly $L_\alp$ is the set
of isolated points of $X^{(\alp)}$.

Since the collection of $X^{(\alp)}$s forms a decreasing sequence of
closed subsets of $X$, if $X$ is a compact scattered space, then it
has a non-empty finite top level $X^{(\gm)}=L_\gm$.

We endow an ordinal (regarded as the set of its own predecessors)
with the interval topology generated by its natural order. With this
topology every ordinal is a scattered space.

The standard set-theoretic notation for the first infinite ordinal,
i.e. the set of all natural numbers, is $\w$. The ordinal $\w+1$,
then, is the set of all ordinals less than \emph{or equal to} $\w$,
so $\w+1$ is the set consisting of $\w$ together with all natural
numbers. Then $\w+1$ with its order topology is homeomorphic to the
convergent sequence $S_0=\{0\}\cup\{1/n:0<n\in\N\}$ with the usual
topology inherited from the real line.  In fact every countable
ordinal is homeomorphic to a subset of $\Q$. The next limit ordinal
is $\w+\w=\w\cdot 2$.  The space $\w\cdot 2+1$ consists of all
ordinals less than or equal to $\w\cdot 2$, i.e. all natural
numbers, $\w$, the ordinals $\w+n$ for each $n\in\N$ and the limit
ordinal $\w\cdot 2$. The set $\w\cdot 2 +1$ with its order topology
is homeomorphic to two disjoint copies of $S_0$. For each $n\in\N$,
the ordinals $n$ and $\w+n$ ($0<n$) have scattered height $0$ in. On
the other hand, $\w$ and $\w\cdot 2$ have scattered height $1$,
corresponding to the fact that $0$ is a limit of isolated points in
$S_0$ but is not a limit of limit points in $S_0$.
The ordinal space $\w^2+1$ consists of all
ordinals less than or equal to $\w^2$ (namely: $0$; the successor
ordinals $n$ and $\w\cdot n+j$, for each $j,n\in\N$; the limit ordinals
$\w\cdot n$, for each $n\in\N$; and the limit ordinal $\w^2$).
With its natural order
topology, $\w^2+1$ is homeomorphic to the subset of the real line
$S=\{0\}\cup\bigcup_{n\in\N}S_n$ defined in the Introduction.
In this case,
the ordinals $\w\cdot n$, $\in\N$, which have scattered height
$1$, correspond to the points $1/n$, which are limits of isolated
points $1/n+1/k$ but not of limit points. The ordinal $\w^2$ has scattered
height $2$ and corresponds to the point $0$, which is a limit of the
limit points $1/n$.

In general, the ordinal space $\w^\alp\cdot n+1$ consists of
$n$ copies of the space $\w^\alp+1$, which itself consist of a
single point with limit type $\alp$ as well as countably many points
of every limit type $\beta$ with $\beta<\alpha$.
It is a standard topological fact that every countable, compact
Hausdorff space $X$ is not only scattered, but homeomorphic to a
countable successor ordinal of the form $\w^\alp\cdot n+1$ for some
countable ordinal $\alp$.
Of course every countable compact metric space is also homeomorphic to a
subset of the rationals and, in this context, we can interpret
the statement that $X\simeq \w^\alp\cdot n+1$ as notation to indicate that
$X$ is homeomorphic to a compact subset of the rationals with $n$ points of
highest limit type $\alp$. All such subsets of $\Q$ are in fact homeomorphic.

\medskip

We will construct countable postcritical $\w$-limit sets for l.e.o.
unimodal maps of the interval with the property that the associated
set of inhomogeneities, $\C I$, is also countable.  In this
construction we make extensive use of \emph{well-founded trees} to
construct appropriate kneading sequences, a
technique we applied in \cite{goodknightraines} . For completeness,  we
briefly describe the construction of a well founded
tree of height $\alp$ for each countable ordinal $\alp$. Such trees
have the remarkable property that they are countable and have finite
length branches but can have height $\alp$ for any countable ordinal
$\alp$. For more details we refer the reader to \cite[I.2]{kechris}.

Let $\C A$ be a countably infinite set of symbols and let $\C
A^{<\N}$ be the set of all finite sequences of elements of $\C A$.
Given two elements $s,t\in \C A^{<\N}$ we say that $t\aux s$ if and
only if $s$ is an initial segment of $t$, i.e. if and only if
$t=(t_1,t_2,\cdots,t_n)$ and $s=(t_1,\cdots,t_m)$ for some $m<n$. If
$n=m+1$, then $t$ extends $s$ by one symbol and we write $t\lessdot
s$. If $s=(s_1,\cdots,s_m)$ and $t=(t_1,\cdots,t_n)$, then we denote
$(s_1,\cdots,s_m,t_1,\cdots,t_n)$ by $s\,t$.

A subset $\C T$ of $\C A^{<\N}$ is said to be a $\emph{tree}$ on $\C
A$ if it is closed under initial segments, i.e whenever $t\in \C T$
and, for some $s\in \C A^{<\N}$, $s$ is an initial segment of  $t$,
then $s\in \C T$.  Since the null sequence $\sartre$ is an initial
segment of any sequence, $\sartre$ is the top element of every tree
on $\C A$.

An \emph{infinite branch} in $\C T$ is an infinite sequence
$b=(b_1,b_2,b_3,\cdots)$ of elements from $\C A$ such that
$(b_1,\cdots,b_n)\in\C T$ for all $n\in \N$. If $\C T$ has no
infinite branches, then the relation $\aux$ is well-founded (i.e.
has no infinite descending chains) and $\C T$ is said to be a
\emph{well-founded} tree.

We can inductively associate a well-defined ordinal height $\hi(s)$
to each element $s$ of a well-founded tree $\C T$ by declaring
$$
\hi(s)=\sup\big\{\hi(t)+1: t\in\C T\text{ and }t\aux s\big\}
$$
and associate to each well-founded tree $\C T$ a well-defined height
$\Hi(\C T)=\hi\big(\sartre\big)$. Clearly, if $t\aux s$, then
$\hi(t)<\hi(s)$, $\hi\big(\sartre\big)>\hi(s)$ for any $\sartre\neq
s\in \C T$ and if $s\in \C T$ has maximal length, then $\hi(s)=0$.

Trees of height $\alp$ can be defined recursively. Let $s_a$ be the
singleton sequence $(a)$ for some $a\in\C A$. Obviously $\C
T_0=\{\nowt\}$ is a tree of height $0$ on $\C A$. So suppose that
$\alp=\bt+1$ and let us assume that there is a tree $\C T_\bt$ on
$\C A$ of height $\bt$. Since $\C A$ is infinite, there is, in fact,
a countably infinite family of disjoint trees $\{\C U_n:n\in\N\}$
each order isomorphic to $\C T_\bt$. Define
$$
\C T_\alp=\big\{\sartre\big\}\cup\big\{s_at:t\in\C U_n,
n\in\N\big\}.
$$
Clearly $\C T_\alp$ is a well-founded tree on $\C A$. Moreover
$\operatorname{ht}_{\C T_\alp}(s_at)=\operatorname{ht}_{\C U_n}(t)$
for every $t\in U_n$ and $n\in \N$, so $\Hi(\C T_\alp)=\bt+1=\alp$.

Now suppose that $\alp$ is a limit ordinal and that for every
$\bt<\alp$ there is a tree $\C T_\bt$ of height $\bt$ on $\C A$.
Again, since $\C A$ is countably infinite, we may assume that $\C
T_\bt$ and $\C T_\gm$ are disjoint whenever $\bt\neq \gm<\alp$.
Define
$$
\C T_\alp=\big\{\sartre\big\}\cup\big\{s_at:t\in\C T_\bt,
\bt<\alp\big\}.
$$
Again it is clear that $\C T_\alp$ is a well-founded tree and that
$\Hi(\C T_\alp)=\alp$.

Notice that, as constructed, if $t\in \C T_\alp=\C T$ for some
$\alp$ and $\hi(t)=\bt$ then if $\gm=\bt+1$, there are infinitely
many $s\in \C T$ such that $\hi(s)=\bt$ and $s\lessdot t$ and if
$\gm$ is a limit, then for each $\bt<\gm$, there is some $s\lessdot
t$ such that $\hi(s)=\bt$.

\section{The Correspondence Theorem}\label{main}

\begin{theorem}\label{main}
  Let $f$ be a unimodal, l.e.o. function with critical point $c$, and
let $\C I=\invlim{\w(c)}{f|_{\w(c)}}$ be the set of folding points
in $\ilim{f}$.  Suppose that $\C I$ is countably infinite. Then
$$\w(c)\simeq \w^{\alp+1}\cdot n+1\qquad\text{ and }\qquad
\C I\simeq \w^{\bt+1}\cdot m+1,$$ for some countable ordinals
$\alp\le \bt$ and natural numbers $n$ and $m$. Moreover
\begin{enumerate}
    \item If $\alp=0$ then $\bt=0$ and $n=m$.

    \item If $1=\alp=\bt$ and either:\begin{enumerate}
    \item $n=1$ then $m$ can be any integer $\ge 1$ except $2$;
    \item $n=2$ then $m$ can be any integer $\ge 2$ except $4$; or
    \item $n>2$ then $m$ can be any integer $\ge n$.\end{enumerate}

    \item If $1=\alp<\bt$ and either: \begin{enumerate}
      \item $n=1$ then $m$ can be any integer $\ge1$ except $2$;
      \item $n=2$ then $m$ can be any integer $\ge2$ except $4$;
      \item $n=4$ then $m$ can be any integer $\ge1$ except $2$; or
      \item $n\nin\{1,2,4\}$ then $m$ can be any integer $\ge1$.
      \end{enumerate}
    \item If $1<\alp=\bt$ and either: \begin{enumerate}
      \item $n=1$ then $m$ can be any integer $\ge 1$ except $2$;
      \item $n=2$ then $m$ can be any integer $\ge 2$ except $4$; or
      \item $n\ge2$ then $m$ can be any integer $\ge n$.
    \end{enumerate}
  \item If $1<\alp<\bt$ then $n$ and $m$ can be any positive integers.
    \end{enumerate}
\end{theorem}

It is well known that if $\w(c)$ is finite then $\C I\simeq \w(c)$, \cite{bruin}.  For the first case, $\alp=0=\bt$ and $n=m$ see \cite{rainestopproc}.  The fact that $\alp$ and $\bt$ are countable and that $m$ and $n$
are finite follows from the discussion of countable compact metric
spaces in Section \ref{count}. The fact that the power of $\w$ in
each case is a successor ordinal ($\alp+1$, $\bt+1$) follows from
Theorem 3.3 of \cite{goodknightraines}. The  rest of the proof follows from
Theorems \ref{UVW}, \ref{UW}, \ref{restriction}, \ref{height 1},
\ref{height 2}, \& \ref{equal height}.

\section{The Constructions} \label{constructions}

In this section we construct two `types' of kneading sequences, $K$ and $K'$ that will give us Theorems \ref{UVW} \& \ref{UW}.

\begin{theorem}\label{UVW}
  Let $1\le \bt$ be countable ordinal and $n,m\in \N$.
  Then there is a unimodal, l.e.o. function $f$ with critical point $c_f$ and kneading sequence $K$ such that $\w(c_f)\simeq\w^{2}\cdot n+1$ and $\C I_f\simeq\w^{\bt+1}\cdot m+1$, in each of the following situations:
  \begin{enumerate}
  \item $1=\bt$ and either
  \begin{enumerate}
  \item $n=1$ and $m\neq2$,
  \item $n=2$ and $2\le m\neq4$, or
  \item $2<n\le m$.
 \end{enumerate}
  \item $1<\bt$ and either
  \begin{enumerate}
  \item $n=1$ and $m\neq2$,
  \item $n=2$ and $2\le m\neq4$,
  \item $n=4$ and $m\neq 2$, or
  \item $n=3$ or $4<n$ and $m\in\N$.
 \end{enumerate}
 \end{enumerate}
\end{theorem}

\begin{theorem}\label{UW}
  Let $2\le \alp\le\bt$ be countable ordinals and $n,m\in \N$.
  Then there is a unimodal, l.e.o. function $g$ with critical point $c_g$ and kneading sequence $K'$ such that $\w(c_g)\simeq\w^{\alp+1}\cdot n+1$ and $\C I_g\simeq\w^{\bt+1}\cdot m+1$, in each of the following situations:
  \begin{enumerate}
    \item $\alp=\bt$ and
    \begin{enumerate}
    \item $n=1$ and $m\neq2$;
  \item $n=2$ and $2\le m\neq4$, or;
  \item $2<n\le m$.
    \end{enumerate}
  \item $\alp<\bt$.
  \end{enumerate}
\end{theorem}

Before we prove these theorems we detail the construction of the kneading sequences $K$ and $K'$.

Let $\alp$ and $\bt$ be countable ordinals such that $1\le\alp\le
\bt$.  Let $\C T_{\alp}$ be the countable well-founded $\alp$ tree
constructed as in the previous section and let $\phi_\alp$ be a
bijection from $\C T_\alp$ to $\mathbb E$, the set of even natural numbers, so that $\phi_\alp$ labels each
node of $\C T_\alp$ uniquely with an even natural number. Since every path
through $\C T_\alp$ is finite, there are countably many paths
through $\C T_\alp$ of the form $t_n\lessdot
t_{n-1}\lessdot\dots\lessdot t_1\lessdot(\,)$ and each of these
paths is uniquely identified by its path-labeling, the sequence of
even natural numbers $\big(\phi_\alp(t_1),\dots\phi_\alp(t_n)\big)$. Let
$\Lambda$ denote the set of all such sequences of path-labels and
index $\Lambda$ as $\{\lmb_i:i\in\N\}$, where
$\lmb_i=(\lmb_{i1},\dots,\lmb_{in_i})$. Similarly if $\phi_\bt$ is a
bijection from $\C T_\bt$ to $\mathbb O$, the set of odd natural numbers, we obtain a countable collection
$\Gamma=\{\gm_i\}$ of path-labels where
$\gm_i=(\gm_{i1},\dots,\gm_{i,m_i})$ and $\gm_{i1}=\phi_\bt(t)$ for
some $t\lessdot (\,)\in\C T_\bt$.

Let $B$, $U$, $V$ and $W$ be any finite words in $0$ and $1$ such
that each contains at least one $1$.  Let $A=10^r1$, where $r$ is such that
$0^r$ is longer than any string of $0$s that can occur in any
concatenation of $B$s, $U$s, $V$s and $W$s.

For each $\gm_i\in \Gamma$ and $\lmb_i\in \Lambda$, let
\begin{align*}\
V_i&=BV^{\lmb_{i1}}BV^{\lmb_{i2}}B\dots BV^{\lmb_{in_i}}B,\\
U_i&=BU^{\lmb_{i1}}BU^{\lmb_{i2}}B\dots BU^{\lmb_{in_i}}B,\\
\hat{U}_i&=BU^{\gm_{im_i}}BU^{\gm_{i m_i-1}}B\dots BU^{\gm_{i1}}B\\
W_i&=BW^{\gm_{im_i}}BW^{\gm_{i m_i-1}}B\dots BW^{\gm_{i1}}B\\
\end{align*}
(note that the exponents in the $\hat{U}_i$ and $W_i$ are in the reverse order). Let $\{C_k:k\in\N\}$
list the set $\{U_i:i\in\N\}$ in such a way that each $U_i$ occurs
infinitely often and let $\{D_k:k\in\N\}$ list $\{V_i:i\in\N\}$ and
$\{E_k:k\in\N\}$ list $\{\hat{U}_i:i\in\N\}$ and $\{F_k:k\in \N\}$ list $\{W_i:i\in \N\}$ in similar fashion. Let
\begin{align*}
X_k&=U^kC_kU^k,\\
Y_k&=V^kD_kV^k,\\
Z_k&=U^kE_kW^k.\\
X_k'&=X_k.\\
Z'_k&=W^kF_kW^k.\\
\end{align*}

We define two kneading
sequences of l.e.o. unimodal maps (in fact, they are the kneading
sequences of tent maps):
\begin{align*}
K&=AAX_1Y_1Z_1X_2Y_2Z_2X_3Y_3Z_3\dots\\
K'&=AAX'_1Z'_1X'_2Z'_2X'_3Z'_3\dots\\
\end{align*}

We begin with a description of the structure of the postcritical omega-limit set and set of inhomogeneities for a tent map, $f$, with kneading sequence $K$.  Then we give the same information for a tent map, $g$, with kneading sequence $K'$.  Let $c_f$ be the critical point for the map $f$, and let $\C I_f$ be the set of inhomogeneities in $\ilim{f}$.  Similarly let $c_g$ be the critical point for $g$, and let $\C I_g$ be the set of inhomogeneities for $\ilim{g}$.

\begin{claim} \label{biinf f}The full itineraries that correspond to points in $\C I_f$ are the following and all of their shifts:
\begin{enumerate}
\item $U^\Z$, $V^\Z$ and $W^\Z$;
\item $W^{-\infty}.U^\infty$, $U^{-\infty}.V^\infty$ and
$V^{-\infty}.U^\infty$;
\item $U^{-\infty}.BU^\infty$, $V^{-\infty}.BV^\infty$ and
$U^{-\infty}.BW^\infty$;
\item $U^{-\infty}BU^{\lmb_{i1}}BU^{\lmb_{i2}}B\dots BU^{\lmb_{in_i}}B.U^\infty$
and\\
$V^{-\infty}BV^{\lmb_{i1}}BV^{\lmb_{i2}}B\dots
BV^{\lmb_{in_i}}B.V^\infty$,
for each $\lmb_i\in\Lambda$, and\\
$U^{-\infty}BU^{\gm_{im_i}}BU^{\gm_{i2}}B\dots
BU^{\gm_{i1}}B.W^\infty$, for all $\gm_i\in\Gamma$.
\end{enumerate}
\end{claim}

\begin{proof} The full itineraries of points in $\C I_f$ are exactly those bi-infinite strings
for which every central segments occurs infinitely often in $K$.
Since each $U_i$, $V_i$ and $\hat{U}_i$ occurs infinitely often in the
enumeration of the $C_k$, $D_k$ and $E_k$, this means that all full
itineraries arise from central segments that occur in the words
$$
W^{k-1}U^kC_kU^kV^kD_kV^kU^kE_kW^kU^{k+1}
$$
as $k\to \infty$.

Clearly the terms $W^{k-1}U^k$, $U^kV^k$, and $V^kU^k$ give rise to
$W^{-\infty}.U^\infty$, $U^{-\infty}.V^\infty$,
$V^{-\infty}.U^\infty$ $U^\Z$, $V^\Z$ and $W^\Z$. Since $k\to
\infty$, any other full itinerary must arise from a subsequence of
the terms $(U^kC_kU^k)_{k\in\N}$ or from a subsequence of
$(V^kD_kV^k)_{k\in \N}$ or from a subsequence of $(U^kE_kW^k)_{k\in \N}$.

Let us consider the full itineraries that we obtain from the
$(U^kC_kU^k)_{k\in \N}$. Since every path in $\C T_\alp$ is finite, if
$t_n\lessdot\dots\lessdot t_1$ is a maximal path, then there is no
$\lmb\in\Lambda$ which extends the path-label
$\big(\phi_\alp(t_1),\dots,\phi_\alp(t_n)\big)$. This implies that
the number of times $B$ occurs in any full itinerary is finite.  Suppose that some full itinerary has at least two $B$s.  Then this full itinerary either has $BU^pB$ or $BV^pB$ for some even $p$ or $BU^qB$ for some odd $q$.
First suppose we are in the case that the full itinerary contains $BU^pB$ and $p$ is even. Since $\phi_\alp$ is a bijection onto the set of even integers, this uniquely identifies a node, $t_p$, in the tree $\C T_\alp$.  Hence there is a path in $\C T_\alp$ that contains $t_p$ and has path-label $\lmb_i=(\lmb_{i1},\lmb_{i2}, \dots \lmb_{i n_i})$ such that this full itinerary is of the form
$$U^{-\infty}BU^{\lmb_{i1}}BU^{\lmb_{i2}}B\dots BU^{\lmb_{in_i}}BU^\infty.$$
Moreover, for each path-label $\lmb_i=(\lmb_{i1},\dots ,\lmb_{in_i})\in\Lambda$ we get such a full itinerary.  It is clear that the only new full itineraries formed as
$i\to \infty$ are $U^{-\infty}BU^\infty$ and $U^\Z$.

If instead we are in the case that $BV^pB$ occurs in the full itinerary with $p$ even, an exactly similar argument shows that it is of the form
$$V^{-\infty}BV^{\lmb_{i1}}BV^{\lmb_{i2}}B\dots BV^{\lmb_{in_i}}BV^\infty,$$
for some $\lmb_i\in\Lambda$.  Then as $i\to \infty$ we get the limit points $V^{-\infty}BV^\infty$ and $V^\Z$.

Finally, if we are in the case that $BU^qB$ occurs in the full itinerary with $q$ odd, then since $\phi_\beta$ is a bijection from the tree $\C T_\beta$ onto the odd integers, this uniquely identifies a node $t_q$ in $\C T_\beta$.  A similar argument shows that the full itinerary is of the form
$$U^{-\infty}BU^{\gm_{im_i}}BU^{\gm_{i2}}B\dots BU^{\gm_{i1}}BW^\infty,$$
for some $\gm_i=(\gm_{i1},\dots ,\gm_{im_i})\in\Gamma$.  Again as $i\to \infty$ the limits that occur are $U^{-\infty}BW^\infty$, $U^\Z$ and $W^\Z$.
\end{proof}

The following claim is now immediate.

\begin{claim}\label{itinwcf} The right itineraries corresponding to points $x\in \w(c_f)$ are the following and all of their shifts:
\begin{enumerate}
\item $U^\infty$, $V^\infty$ and $W^\infty$;
\item $W^{n}U^\infty$, $U^{n}V^\infty$ and
$V^{n}U^\infty$, for each $n\in \N$;
\item $U^{n}BU^\infty$, $V^{n}BV^\infty$ and
$U^{n}BW^\infty$, for each $n\in\N$;
\item \begin{enumerate}
\item $U^{n}BU^{\lmb_{i1}}BU^{\lmb_{i2}}B\dots BU^{\lmb_{in_i}}BU^\infty$,
for each $\lmb_i\in\Lambda$, $n\in\N$,
\item $V^{n}BV^{\lmb_{i1}}BV^{\lmb_{i2}}B\dots BV^{\lmb_{in_i}}BV^\infty$,
for each $\lmb_i\in\Lambda$, $n\in\N$,
\item $U^{n}BU^{\gm_{im_i}}BU^{\gm_{i2}}B\dots BU^{\gm_{i1}}BW^\infty$,
for each $\gm_i\in\Gamma$, $n\in\N$.
\end{enumerate}
\end{enumerate}
\end{claim}

For finite words $A$ and $B$ we say $A\sim B$ if some shift of
$A^\Z$ is equal to a shift of $B^\Z$.

\begin{claim}\label{height w(c)}
$\w(c_f)\simeq\w^{\alp+1}\cdot n+1$ where  $n=|U|$ if $U\sim V$ and
$n=|U|+|V|$ otherwise.
\end{claim}

\begin{proof} Notice that the only itineraries of the form $WU^kD$, $UV^kD$, $VU^kD$
and $BW^kD$, for some $D$, are $WU^\infty$, $WU^\infty$,
$UV^\infty$, $V^\infty$ and $BW^\infty$.  Therefore, $W^{n}U^\infty$,
$U^{n}V^\infty$ $V^{n}U^\infty$ are isolated for all $n>0$.
Similarly, any word ending $BW^\infty$ is isolated.  In particular,
$$U^{n}BU^{\gm_{im_i}}BU^{\gm_{i2}}B\dots BU^{\gm_{i1}}BW^\infty$$ is isolated for any
$\gm_i\in\Gamma$ and $n\in\N$.

Every itinerary of the form $W^kD$ is either of the form
$W^nU^\infty$ or $W^\infty$. Hence $W^\infty$ has limit type $1$.

The itinerary $U^\infty$ is a limit of the sequences $(U^nBU^\infty)_{n\in \N}$ and
$(U^nBW^\infty)_{n\in \N}$ and so the limit type of $U^\infty$ is determined
by the limit types of the points $U^nBU^\infty$ and $U^nBW^\infty$.
Similarly the limit type of $V^\infty$ is determined by the limit
types of the points $V^nBV^\infty$.

We claim that for each $n\ge0$, $U^nBU^\infty$ and $V^nBV^\infty$
have limit type $\alp$. Since each $U^nBW^\infty$ is isolated, this
then implies that $U^\infty$ and $V^\infty$ both have limit type
$\alp+1$.

So consider $U^nBU^\infty$ for some $n\ge0$ (the argument for
$V^nBV^\infty$ is identical). This is a limit of points of the form
(4a) and of points of the form (4c) and their shifts from Claim \ref{itinwcf}. Since points
of the form (4c) are all isolated, we only need to consider points
of the form (4a) that contain two instances of the word $B$ (since
otherwise we are of the form $U^jBU^\infty$). We will show that if
$p\in\N$ then the limit type of any point with right itinerary of
the form $FBU^{p}BU^\infty$, for some finite word $F$, is equal to
the rank of the node $t_p=\phi_\alp^{-1}(p)$ in the tree $\C
T_\alp$. This is enough since node $(\,)$ in the tree $\C T_\alp$
has rank $\alp$, which then implies that $U^nBU^\infty$ has limit
type $\alp$.

The proof is completed easily by induction. For any $\lmb_{ij}$, let
$t_{ij}$ be the node in the tree $\C T_\alp$ such that $\phi_\alp(t_{ij})=\lmb_{ij}$. Note that
the natural number denoted $\lmb_{ij}$ corresponds uniquely to
$t_{ij}$ (even though this natural number may be listed as a label
in path-labels $\lmb_i$ for different $i\in\N$). Suppose that
$t_p=t_{in_i}$ is a terminal node in $\C T_\alp$, and hence has rank
$0$. Then no path-label in $\Lambda$ extends the finite sequence
$\lmb_i=(\lmb_{i1},\dots\lmb_{in_i})$ and so the only point with
right itinerary of the form $FBU^{\lmb_{in_i}}BU^rG$, where $r>0$
and $G$ is a right infinite word, is $FBU^{\lmb_{in_i}}BU^\infty$
itself. Hence this point is isolated and has limit type $0$. Now
suppose that $t_p$ is not a terminal node. Then there is an infinite
set $N\subseteq \N$ such that $t_p$ occurs as the penultimate label
$t_{i(n_{i}-1)}$ in the path-label $\lmb_i$ for each $i\in N$. By
induction, each point with right itinerary
$$
FBU^{p}BU^{\lmb_{in_i}}BU^\infty
=FBU^{\lmb_{i n_i-1}}BU^{\lmb_{in_i}}BU^\infty
$$
has limit type equal to the rank of $t_{in_i}$. Since these are
precisely the nodes in $\C T_\alp$ that are below $t_p$, the limit
type of $FBU^pBU^\infty$ is the same as the rank of $t_p$ and we are
done.
\end{proof}

\begin{claim}\label{height I} If $\alp<\bt$, then $\C I_f\simeq \w^{\bt+1}\cdot m+1$, where $m=|U|$ if $U\sim W$ and $m=|U|+|W|$ otherwise.
If $\alp=\bt$, then then $\C I_f\simeq \w^{\bt+1}\cdot m+1$ where
$$m=\left\{
              \begin{array}{ll}
                |U|, & \hbox{$U\sim V\sim W$;} \\
                |U|+|V|, & \hbox{$U\sim W\text{ or } V\sim W$;} \\
                |U|+|W|, & \hbox{$U\sim V\text{ or } W\sim V$;} \\
                |V|+|W|, & \hbox{$V\sim U \text{ or } W\sim U$;} \\
                |U|+|V|+|W|, & \hbox{otherwise.}
              \end{array}
            \right.$$
\end{claim}

\begin{proof} If a full itinerary contains two $B$s then it must contain $BU^pB$, $BV^pB$ for some even integer $p$ or $BU^qB$ for some odd integer $q$.  (Recall that we labeled the nodes of the tree $\C T_\alpha$ with the even integers while we labeled the nodes of the tree $\C T_\bt$ with odd integers.)

Arguing again by induction,
as in the proof of Claim \ref{height w(c)}, one can easily show that
for any left infinite word $F$ and even integer $p$, the limit type of both points with
full itineraries either $FBU^{p}B.U^\infty$ or $FBV^{p}B.V^\infty$,
are both equal to the rank of the node $t_p=\phi_\alp^{-1}(p)$ in
$\C T_\alp$. The same is true of any shift of these points. We can
also see, by an argument identical to that in the proof of Claim \ref{height w(c)}, that for any right infinite word
$G$ and odd integer $q$, the limit type of points with full itinerary
$U^{-\infty}BU^{q}BG$ is equal to the rank of the node
$\phi_\bt^{-1}(q)$ in $\C T_\bt$.

The points with full itineraries $W^{-\infty}.U^\infty$,
$U^{-\infty}.V^\infty$ and $V^{-\infty}.U^\infty$ are all isolated
since, for example, no other full itineraries have the form $\dots
W.U\dots$.

Points whose full itinerary is a shift of $V^{-\infty}.BV^\infty$ are
limits of points with itineraries that are shifts of full itineraries
of the form $$V^{-\infty}BV^{\lmb_{i1}}BV^{\lmb_{i2}}B\dots
BV^{\lmb_{in_i}}B.V^\infty.$$ Hence points corresponding to shifts of $V^{-\infty}.BV^\infty$ have limit type $\alp$.

Points whose full itinerary is a shift of
$U^{-\infty}.BW^\infty$ and $U^{-\infty}.BU^\infty$ are limits of
points with full itineraries that are shifts of
$$U^{-\infty}BU^{\gm_{im_i}}BU^{\gm_{i2}}B\dots
BU^{\gm_{i1}}B.W^\infty.$$ Hence these points have limit type $\bt$.

It follows then that $V^\Z$ has limit type $\alp+1$, while $U^\Z$
and $W^\Z$ both have limit type $\bt+1$.
\end{proof}

Now we turn our attention to the map $g$ with kneading sequence $K'$.  The next four claims follow by arguments similar to the proofs of Claims \ref{biinf f}, \ref{itinwcf}, \ref{height w(c)} and \ref{height I}.

\begin{claim} The full itineraries that correspond to points in $\C I_g$ are the following and all of their shifts:
\begin{enumerate}
\item $U^\Z$ and $W^\Z$;
\item $W^{-\infty}.U^\infty$ and $U^{-\infty}.W^\infty$;
\item $U^{-\infty}.BU^\infty$ and  $W^{-\infty}.BW^\infty$;
\item $U^{-\infty}BU^{\lmb_{i1}}BU^{\lmb_{i2}}B\dots BU^{\lmb_{in_i}}B.U^\infty$ for each $\lmb_i\in \Lambda$
and\\
$W^{-\infty}BW^{\gm_{im_i}}BW^{\gm_{i2}}B\dots
BW^{\gm_{i1}}B.W^\infty$, for each $\gm_i\in\Gamma$.
\end{enumerate}
\end{claim}

\begin{claim} The right itineraries corresponding to points $x\in \w(c_g)$ are the following and all of their shifts:
\begin{enumerate}
\item $U^\infty$ and $W^\infty$;
\item $W^nU^\infty$ and $U^nW^\infty$ for each $n\in \N$;
\item $U^nBU^\infty$ and  $W^nBW^\infty$ for each $n\in \N$;
\item $U^nBU^{\lmb_{i1}}BU^{\lmb_{i2}}B\dots BU^{\lmb_{in_i}}BU^\infty$ for each $\lmb_i\in \Lambda$, $n\in \N$
and\\
$W^nBW^{\gm_{im_i}}BW^{\gm_{i2}}B\dots
BW^{\gm_{i1}}BW^\infty$, for each $\gm_i\in\Gamma$, $n\in \N$.
\end{enumerate}
\end{claim}

\begin{claim}\label{gheight w(c)}
  $\w(c_g)\simeq \w^{\alp+1}|U|+1$.
\end{claim}

\begin{claim}\label{gheight I}
  If $\alp<\bt$, then $\C I_g\simeq \w^{\bt+1}|W|+1$.  If $\alp=\bt$ then $\C I_g\simeq \w^{\bt+1}\cdot m+1$ where $m=|U|$ if $U\sim W$ and $m=|U|
+|W|$ otherwise.
\end{claim}

We end this section with proofs of Theorems \ref{UVW} \& \ref{UW}.

\begin{proof}[Proof of \ref{UVW}.]
By careful choice of $U$, $V$, and $W$ in the construction of $K$ outlined above, we obtain a kneading sequence for a tent map $f$ with the properties described in the theorem.

 For case (1) we have $\beta=1$.  If $n=1$ and $m\neq 2$, let $U=V=1$.
If also $m=1$ then let $W=1$.  If instead $m>2$, let $W=01^{m-2}$.
Then we see that $\w(c_f)\simeq \w^{2}\cdot +1$ and $\C I_f\simeq
\w^{\bt+1}\cdot m+1$ by Claims \ref{height w(c)} \& \ref{height I}
proving case (1.a). Similarly when $n=2$, let $U=V=01$.  If also
$m=2$ then let $W=01$. If instead $m=3$ then let $W=1$.  If $4<m$
then let $W=01^{m-3}$, which establishes case (1.b). For case (1.c),
$2<n$, let $U=V=01^{n-1}$. If $m=n$ let $W=U$.  Instead suppose that
$m>n$. If $m=n+1$ then let $W=1$.  If $m=n+2$ then let $W=01$.  If
$m>n+2$ then let $W=001^{m-n-2}$.

For case (2) let $1<\beta$.   If $n=1$, let $U=V=1$.  If $m=1$ then
let $W=U$.  If instead $m>2$ then let $W=01^{m-2}$.  This gives case
(2.a).  Let $n=2$ and $U=V=01$.  If $m=2$ let $W=U$.  If $m=3$ let
$W=1$.  If instead $m>4$ let $W=01^{m-3}$.  This gives case (2.b).

For case (2.c), let $n=4$.  If $m=1$ let $U=W=1$ and $V=011$.  If
instead $m=3$ let $U=W=011$ and $V=1$.  To finish the case, let
$U=V=0111$.  If $m=4$ then let $U=W$.  If $m=5$ let $W=1$, but if
$m=6$ let $W=01$.  Finally if $m\ge 7$ let $W=001^{m-n-2}$.

For case (2.d), we start with $n=3$ and $m\ge n$.  Let $U=V=011$ and
if $m=3$ let $W=U$.  If $m=4$ let $W=1$.  If $m=5$ let $W=01$, and
if $m>5$, let $W=001^{m-5}$.  To finish the case that $n=3$, let
$m=2$.  Then chose $U=W=01$ and $V=1$.  If instead $m=1$ let $U=W=1$
and $V=01$.  Next suppose that $n>4$ and $m\ge n$.  Let
$U=V=001^{n-2}$.  If $m=n$ let $U=W$.  If $m=n+1$ let $W=1$, and if
$m\ge n+2$ let $W=01^{m-n-1}$.  Finally, suppose that $n>4$ and
$m<n$.  If $m=1$ let $U=W=1$ and $V=01^{n-2}$.  If $m=2$ let
$U=W=01$ and let $V=001^{n-4}$. If $m\ge 3$ let $U=W=001^{m-2}$ and
$V=01^{n-m-1}$.
\end{proof}

\begin{proof}[Proof of \ref{UW}.]
  Suppose we are in case (1), $\alp=\bt$.  If $n=1$ let $U=1$.  If
$m=1$ let $U=W$.  If instead $m>2$ let $W=01^{m-2}$.  Case (1.a)
follows by Claims \ref{gheight w(c)} \& \ref{gheight I}.  For case
(1.b), let $n=2$ and $U=01$.  Then if $m=2$ let $U=W$, but if $m=3$
let $W=1$.  If instead $m>4$ let $W=01^{m-3}$.  To finish case (1),
let $2<n$.  Then take $U=001^{n-2}$, and if $m=n$ let $W=U$.  If
$m=n+1$ let $W=1$, and if $m>n+1$ let $W=01^{m-n-1}$.

Next suppose we are in case (2), i.e. $\alp<\bt$.  If $n\ge 3$ then
we can chose words $U$ and $W$ with $U\not\sim W$ such that $|U|=n$
and $|W|=m$.  This is also true if $m\ge 3$.  So suppose that
$m,n\le 2$.  If $m=n$ then let $U=V$ is a word with a $1$ in it of
length $n$.  If $m\neq n$ then one of them is $1$, without loss of
generality, $n=1$.  Then let $U=1$ and $W=01$.
\end{proof}

\section{Restrictions on the Limit Complexity of the Set of
Inhomogeneities}\label{restrictions}

In this section we show that the cases for which we did not give
constructions in the previous section cannot exist.  We accomplish
this by describing the restrictions on the limit complexity of the
set $\C I$.  In \cite{goodknightraines}, we proved the following
theorem which gives a restriction on the limit complexity for any
countable $\w$-limit set and set of inhomogeneities.

\begin{theorem}\cite{goodknightraines}\label{restriction}
  Let $f$ be a unimodal, l.e.o. function with critical point $c_f$
such that $\w(c_f)$ and $\C I_f$ are countably infinite.  Then there
are countable ordinals $\alp$ and $\bt$ and positive integers $n$
and $m$ such that $\w(c_f)\simeq \w^{\alp+1}\cdot n+1$ and $\C
I_f\simeq \w^{\bt+1}\cdot m+1$.
\end{theorem}

We state here the main results for this section.

\begin{theorem}\label{height 1}
  Let $n\in \N$.  If $f$ is a unimodal, l.e.o. function with
critical point $c$ and $\w(c)\simeq \w\cdot n+1$, then the set of
inhomogeneities, $\C I$, has $\C I\simeq \w\cdot n+1$.
\end{theorem}

\begin{theorem}\label{height 2}
  Let $1\le \bt$ be a countable ordinal and $n,m\in \N$.  If $f$ is
a unimodal, l.e.o. function with critical point $c_f$ such that
$\w(c_f)\simeq \w^2\cdot n+1$ and $\C I_f\simeq \w^{\bt+1}\cdot m+1$
then:\begin{enumerate}
\item when $\bt=1$,\begin{enumerate}
\item if $n=1$ then $m\neq 2$,
\item if $n=2$ then $2\le m\neq 4$, and
\item if $2<n$ then $n\le m$.
\end{enumerate}
\item when $1<\bt$,\begin{enumerate}
\item if $n=1$ then $m\neq 2$,
\item if $n=2$ then $2\le m\neq 4$, and
\item if $n=4$ then $m\neq 2$.
\end{enumerate}

\end{enumerate}
\end{theorem}

\begin{theorem}\label{equal height}
  Let $2\le \alp$ and $n,m\in \N$.  If $g$ is a unimodal, l.e.o. function with
critical point $c_g$ such that $\w(c_g)\simeq \w^{\alp+1}\cdot n+1$
and $\C I_g\simeq \w^{\alp+1}\cdot m+1$ then \begin{enumerate}
\item if $n=1$ then $m\neq 2$,
\item if $n=2$ then $2\le m\neq 4$, and
\item if $2\le n$ then $n\le m$.
\end{enumerate}
\end{theorem}

\begin{lemma}
  Let $n\in \N$.  Suppose that $f$ is a unimodal, l.e.o. function with critical
 point $c_f$.  Suppose that $\w(c_f)\simeq \w\cdot n+1$.
Then the limit points of $\w(c_f)$ are periodic.
\end{lemma}

\begin{proof}
  This follows since $f(\w(c_f))=\w(c_f)$ and $f$ must map these
  finitely many limit points to themselves (since it is finite-to-one).
\end{proof}

\begin{proof}[Proof of \ref{height 1}.] We know that $c$ is  not recurrent, i.e $c\notin\w(c)$, as otherwise $\w(c)$ would be a Cantor set. Since $f$ is l.e.o.,
therefore, the itinerary map is a homeomorphism from $\w(c)$ to its
image. Let $W_i$, $i\le n$ be a finite word in $\{0,1\}$ such that
the (periodic) limit points of $\w(c)$ have itinerary
  of the form $W_i^\infty$ and their shifts.

 Let $x\in \C I$.  Then consider the full itinerary of $x$,
 $Fi(x)$.  For each $i\in \N$, let $B_i$ be the itinerary of
 $x_i=\pi_i(x)\in \w(c)$.  Since $\w(c)$ is closed, the limit points of $(x_i)_{i\in \N}$ are in $\w(c)$.  Suppose that
 the $x_i$s converge to more than one point, $x$ and $x'$.  Then there are some pair of words $W_i$ and $W_j$ such that the itinerary
 of $x$ is $W_i^\infty$ and the itinerary of $x'$ is $W_j^\infty$.
 Let $(j_n)_{n\in \N}$ and $(i_n)_{n\in \N}$ be defined such that
 $j_{n+1}<i_{n+1}<j_n<i_n$ for each $n\in \N$ and such that
 $I(x_{j_n})=W_j^{m_n}...$  and
 $I(x_{i_n})=W_i^{m_n}...$ where $m_n\to \infty$.  Since the
 sequences $(i_n)_{n\in \N}$ and $(j_n)_{n\in \N}$ are intertwined
 we see that each $I(x_{j_n})=W_j^{m_n} U_{j_n}V_{j_n} W_i^{m_n}...$, where $U_{j_n}\neq W_j$ but $|U_{j_n}|=|W_j|$.
 Since there are only finitely many possibilities for $U_{j_n}$, there
 is some $U\neq W_j$ of the same length as $W_j$ and
 another sequence $(x_{t_n})_{n\in\N}$ with $I(x_{t_n})=W_j U...$.  As $n\to \infty$, the points $x_{t_n}$ converge to a
 limit point $z$ with itinerary $I(z)=W_j U ...$. Clearly $z$ is not periodic which contradicts the
 fact that it is a limit point.  Hence $(x_i)_{i\in \N}$ has a unique limit point $W_i^\infty$, say.
 This implies that the full itinerary of $x$ is of the form
 $Fi(x)=\sh^r(W_i^{-\infty}. V)$ for some $r\in \Z$.  By a similar argument using the sequence
 $(f^i(x))_{i\in\N}$ we see that in fact there is a finite word $U$
 with $Fi(x)=\sh^r(W_i^{-\infty}. U W_j^\infty)$ for some $r\in \Z$ and finite word
 $W_j$
 that gives the itinerary of a limit point of $\w(c)$, where, if  $|U|>0$, then it  does not have $W_i$ as an initial segment
 or $W_j$ as a final segment.

 Suppose that for every $m\in \N$ there is a point, $x_m\in
 \C I$ with $Fi(x_m)=\sh^{r_m}(W_{i}^{-\infty} U_m
 .W_{j}^\infty)$ for some $r_m\in \Z$
 and with $|U_m|\ge m$.  Then let $(y_m)_{m\in \N}\subseteq \w(c)$ such that $I(y_m)=U_m
W_j^{\infty}$.  Let $y\in (y_m)'_{m\in \N}$.  Then by assumption,
$I(y)=W_k^\infty$ for some $1\le k\le n$.  Since $W_i$ is not an
initial segment of $U_m$ we know $k\neq i$. So
$Fi(x_m)=\sh^{r_m}(W_i^{-\infty} W_k^{t_m} U_m'. W_j^\infty)$ for
some $r\in \Z$. Hence there is an infinite collection $(z_m)_{m\in
\N}$ with $I(z_m)=W_i W_k^{t_m} U_m' W_j^{\infty}$.  This converges
to $W_i W_k^\infty$ which is not a limit point for $\w(c)$, a
contradiction. Thus there is some $K$ such that for every $x\in \C
I$ if $Fi(x)=\sh^r(W_i^{-\infty} . U W_j^\infty)$ for some $r\in \Z$
then $0\le |U|\le K$.

We next show that $\C I$ is homeomorphic to $\w(c)$ by showing that
it is a countable collection of isolated points with the same number
of limit points as $\w(c)$.  Clearly, for every $x\in \C I$ with
$Fi(x)=\sh^r(W_i^{-\infty}.W_i^\infty)$ is a limit of points from
$\C I$.  So instead consider $x$ with $Fi(x)=\sh^r(W_i^{-\infty}U .
W_j^\infty)$ for some $r\in \Z$ and some $U$ with $0\le |U|\le K$.
There are only finitely many choices for $U$ and so all of these
countably many points are isolated. Hence $\C I$ is a finite
collection of convergent sequences.
\end{proof}

Next we show that if $\w(c)$ is homeomorphic to some
$\w^{\alpha+1}\cdot n+1$ then $\C I$ cannot be less complicated then
$\w(c)$.  This will follow because the projection map is a
continuous surjection and the following proposition.

\begin{lemma}\label{proj}
  Let $\alpha$ and $\beta$ be countable ordinals, and suppose that $p:\w^\beta\cdot m+1\to \w^\alpha\cdot n+1$ is a
   continuous surjection.  Then $\alpha\le \beta$.  Moreover, if
   $\alpha=\beta$ then $m\le n$.
\end{lemma}

\begin{proof}
  Since $p$ is continuous, if $z\in \w^\alpha\cdot n+1$
  is an isolated point then $U=\{z\}$ is an open set containing $z$.  Hence $p^{-1}(U)$ is open set
  in $\w^\beta\cdot m+1$ therefore it contains an isolated point $z'$ which maps to $z$.  We continue inductively
  to prove that if $x\in \w^\alpha\cdot n+1$ with limit type $\gamma$ then there is a point in the preimage of $x$ that
  has limit type at least $\gamma$.

  Suppose that for all $\lambda<\gamma$ every point in $\w^\alpha\cdot n+1$ with limit type $\lambda$ has a point in
  its preimage with limit type at least $\lambda$.  Let $z\in \w^\alpha\cdot n+1$ such that the limit type of $z$ is $\gamma$.
  Then there is a clopen neighborhood, $U$, of $z$ such that $z$ is the only point in $U$ with limit type $\gamma$ and for any
  $\lambda<\gamma$ there are infinitely many points in $U$ with limit type $\lambda$.  Then $p^{-1}(U)$ is a clopen neighborhood of
  $p^{-1}(z)$ which contains infinitely many points with limit type at least $\lambda$ for all $\lambda<\gamma$.  Hence some
  point in $p^{-1}(U)$ has limit type greater than or equal to
  $\gamma$.  Since this is true for all clopen neighborhoods $U$ of
  $z$ we see that $p^{-1}(z)$ contains a point with limit type
  $\gamma$.  It follows that the points, $v$, with highest limit type in $\w^\alpha\cdot
  n+1$ have a preimage with limit type at least $\alpha$.  Hence
  $\alpha\le \beta$.

  Now suppose that $\alpha=\beta$.  Then there are $n$-many points,
  $v_i$, in $\w^\alpha\cdot n+1$ with highest limit type.  For each
  $1\le i\le n$, $p^{-1}(v_i)$ contains a point, $w_i$, with limit type
  $\alpha$.  Thus in $\w^\beta\cdot m+1=\w^\alpha\cdot m+1$ we have at least $n$-many
  points with limit type $\alpha$.  Hence $n\le m$.
\end{proof}

\begin{lemma}
  Let $f:[0,1]\to [0,1]$ be unimodal with critical point $c$.  Then if $\alpha$ is a countable ordinal and $\w(c)$
  contains a point with limit type $\alpha$, then $\C I$ must contain a point with limit type $\alpha$.
\end{lemma}

\begin{proof}
  This follows from the previous lemma and the fact that $\pi_1$ is a continuous surjection from $\C I$ onto $\w(c)$.
\end{proof}

The next proposition will imply a restriction on the size of the top
level of $\C I$.

\begin{proposition}\label{proj level2}
  Let $1\le\alp\le \bt$ be countable ordinals and $n,m\in \N$.  Let
$f$ be a unimodal, l.e.o. function with critical point $c$. Suppose
that $\w(c)\simeq \w^{\alpha+1}\cdot n+1$.  Suppose that $\C I\simeq
\w^{\beta+1}\cdot m+1$.  Then there is a point $v\in \C I$ with
limit
  type $\beta+1$ such that $\pi_1(v)\in \w(c)$ has limit type
  $\gamma\ge 2$.
\end{proposition}

\begin{proof}  

 Consider the finite set of
  points in $\C I$ with limit type $\beta+1$, $L_{\beta+1}$.  Since $\sh$
  is a homeomorphism on $\ilim{f}$ with $\sh(\C I)=\C I$, we see
  that $\sh(L_{\beta+1})=L_{\beta+1}$.  Hence each point in $L_{\beta+1}$ is a
  periodic point under $\sh$.  Since $f$ is l.e.o. and unimodal we
  see that each point in $L_{\beta+1}$ can be represented by a full
  itinerary of the form $U^{\Z}$ where $U$ is a finite word.

 Let $w\in L_{\beta+1}$.  Then there is a sequence $u_n\in L_{\beta}$ such
  that $u_n\to w$.  Moreover,  for each $n$, the sequence $(\sh^{k}(u_n))_{k\in
  \Z}$ has all of its limit points in $L_{\beta+1}$ which is a finite
  set of periodic points.  We claim that this implies that the full
  itinerary of $u_n$ is $U^{-\infty} B V^{\infty}$ where $U^\Z$ and
  $V^{\Z}$ are each a full itinerary of a point in $L_{\beta+1}$ and $B$ is a finite word (possibly empty).  To
  see this notice that as $k\to -\infty$ there is a subsequence,
  $(k_i)_{i\in \N}$ such that $\sh^{k_i}(u_n)\to u\in L_{\beta+1}$ where
   the full itinerary of $u$ is of the form $U^\Z$ we see that
  for some finite words, $B_n$ with the property that the first $|U|$ segment
   of $B_n$ is not $U$, and right-infinite word, $B_0$, and integers $m_n$ the full itinerary of $u_n$ is
  $$\dots U^{m_n}B_nU^{m_{n-1}}B_{n-1}\dots U^{m_1}B_1B_0$$  or
  $$U^{-\infty} B_1.B_0$$  It might be the case that $B_1$ is the empty word.  In the first case by considering the
  shifts of $u_n$ that correspond with $$\dots
  B_{n+1}U^{m_n}.B_n\dots$$ we get a limit point, $v\in L_{\beta+1}$, of the form
  $U^{-\infty}.B$ where the first $|U|$-length segment of $B$ is not
  $U$.  Now by considering shifts of that word with full itinerary
  $U^{-\infty}.U^NB$ we see that the point $u$ with full itinerary
  $U^\Z$ is actually a limit point of $L_{\beta+1}$ and hence is in
  $L_{\beta+2}$.  This contradicts the assumption that $\beta+1$ is
  the highest limit type in $\C I$.  Thus the full itinerary of
  $u_n$ is $$U^{-\infty}B_1.B_0$$ where $B_1$ could be empty.  By a similar argument for the
  positive shifts of $u_n$, we see that in fact the full itinerary
  of $u_n$ is of the form $$U^{-\infty}B_1B_0'V^{\infty}$$ where
  $B_0'$ is a finite (perhaps empty) subword of $B_0$ with the property that the
  last $|V|$-length segment of $B_0'$ is not $V$ if it is non-empty, and the origin could be in between $V$s $U$s or $B_0'$ and
  $B_1$.  Notice that each $B_0'$ and $B_1$ depends upon the point
  $u_n$ so there are potentially infinitely many different such
  words.  However if these central words are nonempty and grow with unbounded length in $n$ then
  they will limit to something in $L_{\beta+1}$.  This
  will lead us to a contradiction as above.  Thus there are only finitely many different words $B_1B_0'$ that are
  possible.  So we lose no generality in assuming that they are all
  the same, $B$.  Then the $u_n$s are just points with full
  itinerary $$U^{-\infty}BV^\infty$$

 Since $\beta\ge \alpha\ge 1$ we know that the limit
  type of $u_n$ is at least $1$.  So for each $n\in \N$ let
  $v_{k,n}$ be a sequence of points in $\C I$ such that $v_{k,n}\to
  u_n$ as $k\to \infty$.  We have the following possibilities for
  the full itinerary of $v_{k,n}$:\begin{enumerate}
    \item $B_kU^{N_k}B V\infty$; \item $B_kU^{N_k}BV^{M_k}C_k$;
    \item $U^{-\infty}BV^{M_k}C_k$
  \end{enumerate}
 where the last $|U|$-length segment of $B_k$ is not $U$ and the
 first $|V|$-length segment of $C_k$ is not $V$ and $M_k, N_k\to
 \infty$ as $k\to \infty$.

In cases $(2)$ and $(3)$ we have points in $\w(c)$ with itinerary
$U^{N_k}BV^{M_k}C_k$.  Fix $n\in \N$.  Then for all $M_k\ge n$ we
also have a point in $\w(c)$ with itinerary $U^nBV^{M_k}C_k$.  So as
$k\to \infty$ these points converge to a point with limit type at
least $1$ with itinerary $U^nBV^\infty$. But these converge to a
point with limit type at least $2$ with itinerary $U^\infty$.  Hence
there is a point, $u\in \C I$, with limit type $\beta+1$ such that
$\pi_1(u)$ has limit type greater than $1$.  So suppose that the
itineraries of $v_{n,k}$ satisfy case $(1)$ above.

Now consider the shifts of $v_{n,k}$ which have full itinerary
$B_k.U^{N_k}\dots$.  As $k\to \infty$, $B_k\to D$ where the last
$|U|$ length segment of $D$ is not $U$.  Thus these shifts of
$v_{n,k}$ converge to some point, $x$, with full itinerary
$D.U^\infty$ in $L_{\beta}$.  The previous argument regarding the
full itinerary of $u_n$ holds for any point in $L_{\beta}$.  Hence
either there is some non-empty word $C$ such that the full itinerary
of $x$ is $Y^{-\infty} C .U^\infty$, or the full itinerary of $x$ is
of the form $Y^{-\infty}.U^\infty$ where $Y^\Z$ corresponds to a
point in $L_{\beta+1}$.

Suppose that the full itinerary of $x$ is $Y^{-\infty} C.U^\infty$.
Then since $B_n\to D=Y^{-\infty}C$ for every $m\in \N$ there is a
$P_m\in \N$ such that for all $N\ge P_m$, $$B_N=B_N'Y^mC$$ where
$B_N'\neq Y^{-\infty}$.  So we see that in fact we have for each
$m\in \N$ infinitely many $v_{n,k}$s with full itinerary
$$B_N'Y^mCU^{N_k}BV^\infty$$  Thus for each $m$ there is an infinite
collection of points in $\w(c)$ with itinerary
$Y^mCU^{N_k}BV^\infty$.  Then as $k\to \infty$ we get limit points
in $\w(c)$ with itinerary $Y^mCU^\infty$.  These points have limit
points with itinerary $Y^\infty$.  Since $Y^\Z$ is the full
itinerary of a point, $y$, in $\C I$ with limit height $\beta+1$, we
see that in this case there is a point, namely $y$, in $L_{\beta+1}$
such that $\pi_1(y)$ is not limit type $0$ or $1$.

Suppose instead now that the full itinerary for $x$ is
$Y^{-\infty}.U^\infty$.  Then, again, since $B_n\to D=Y^\infty$ for
every $m$ there is a $P_M\in \N$ such that for all $N\ge P_m$
$$B_N=B_N'D_NY^m$$ where $D_N\neq Y$ but $|D_N|=|Y|$.  So for each $m$ we
have infinitely many $v_{n,k}$s with full itinerary
$$B_N'D_NY^mU^{N_k}BV^\infty$$  If $Y\neq U$ then we have points in
$\w(c)$ with itinerary $Y^mU^{N_k}BV^\infty$ which converge to limit
points with itinerary $Y^mU^{\infty}$ which in turn converge to
points with itinerary $Y^\infty$.  Again we get a point $y\in \C I$
with limit type $\beta+1$ that has the property that $\pi_1(y)$ does
not have limit type $0$ or $1$ in $\w(c)$.  If instead though we
have $Y=U$ then since $|D_N|=|Y|$ without loss of generality we can
assume that all of the $D_N$s are the same finite word $D$.  By
shifting again we see that there is a sequence of points with full
itinerary $B_N'D.U^{m+N_k}\dots$.  These converge to a point in
$L_{\beta}$ with full itinerary $BD.U^\infty$ where $D\neq U$ and
$|D|=|U|$.  Since this point is in $L_{\beta}$ there is some $Z$
such that $B=Z^{-\infty}E$.  Since $B_N'\to B$, for every $m$ there
is an $P_m\in \N$ such that for $N\ge P_m$, $B_N'=B_N''Z^mE$, where
$B_N''\neq Z^{-\infty}$ and $Z^\Z$ is the full itinerary of a point
in $L_\beta$.  This implies that for every $m$, there are infinitely
many $v_{n,k}$s with full itinerary $$B_N''Z^mEDU^{Q_k}BV^{\infty}$$
where $Q_k\to \infty$.  This implies that in $\w(c)$ there are
infinitely many points with itinerary $Z^mEDU^{Q_k}BV^\infty$ which
limit to points with itinerary $Z^mEDU^\infty$.  These points have
limit type at least $1$ and they limit to points with itinerary
$Z^\infty$ which has limit type at least $2$.  Hence there is a
point, $z$, in $\C I$ with limit type $\beta+1$ with the property
that the limit type of $\pi_1(z)$ is at least $2$.
\end{proof}

\begin{proof}[Proof of \ref{height 2}.]
  Since $\w(c_f)$ has limit height 2, some point in $\C I_f$ with limit type
$\beta+1$ must project into the top level of $\w(c_f)$, Proposition
\ref{proj level2}.  Notice that there is only one point in
$\ilim{f}$ that is fixed by the shift map.  This point has full
itinerary $1^\Z$.  So suppose that $m=2$.  Then we know that the
points in $\C I_f$ with limit type $\bt$ must be period $2$ under
$\sh$.  Thus one of them projects to some $x\in \w(c_f)$ with limit
type $2$.  This point cannot have itinerary $1^\infty$.  Thus $n\neq
1$.  This establishes cases (1.a) and (2.a).

Similarly, there are only two points in $\ilim{f}$ that have period
$2$ under $\sh$.  These points have full itinerary $(10)^\Z$ or
$(01)^\Z$.  Also there are only two points in $[0,1]$ with period
$2$ under $f$, and they have itinerary $(01)^\infty$ or
$(10)^\infty$.  Therefore if $m=4$ then $n\neq 2$.  Also if $m=2$
then $n\neq 4$.  Lastly, if $\beta=1$ then by Lemma \ref{proj} we
see that $m\ge n$.  This establishes the remaining cases of the
theorem.
\end{proof}

\begin{proof}[Proof of \ref{equal height}.]  Since $\C I_g$ and $\w(c_g)$
have the same limit height, $n\le m$ by Lemma \ref{proj}.  This
implies case (3) immediately.  By similar reasoning as above if
$m=2$ then $n\neq 1$, and if $m=4$ then $n\neq 2$.
\end{proof}

This completes the proofs of the theorems needed to prove the main
theorem of the paper, Theorem \ref{main}.

\section{Countable Postcritical $\w$-limit Sets with Uncountable
Inverse Limit Spaces}

We end the paper with a few examples indicating that $\C I_f$ can be
much more complicated than $\w(c_f)$.  We outline a construction of
maps for which $\w(c_f)$ has limit type $2$ but $\C I_f$ is a Cantor
set or a Cantor set with a countable collection isolated points.

Consider the full binary tree.  We will label the vertices (starting
at the $0$th level) left to right following the rule that the $j$th
vertex in the $n$th level is labeled $2^n+j$ for $0\le j<2^n$.  Let
$\C L$ be the set of all infinite walks through the labeled tree
starting at the first vertex.  Notice that $\C L$ has
 uncountably many (actually $2^\w=\mathfrak c$ many) points and with the usual
topology on $\C L$, agreement along initial segments, $\C L$ is a
Cantor set. Let $\Gamma$ be the set $\{(\dots \gamma_{2}, \gamma_{1},
\gamma_0): (\gamma_0, \gamma_1, \gamma_2\dots)\in \C L\}$.  Clearly
$\Gamma$ is also a Cantor set.

 Let $B_i=101^i$ for all $i\nin\{2^n\}_{n\in \N}$ and
let $B_{2^n}=1^{2^n}$ for $n\in \N$.  Let $A=10001$, and let
$(n_i)_{i\in \N}$ be an increasing sequence of positive integers.
Let
$$K=AA1^{n_1}B_11^{n_2}B_2B_11^{n_3}B_3 B_1 1^{n^4} B_4 B_2
B_11^{n_5}B_5B_2B_11^{n_6}B_6B_3B_1\dots$$ where the indices on
the strings of $B_j$'s follow
longer and longer labeled paths backwards through the full binary tree.  Clearly $K$ is the kneading sequence of a tent map $f$ with critical point $c$.

\begin{theorem}
  The point $x\in\ilim{f}$ is in  $\C I_f$ if, and only if, the full
itinerary of $x$ is a shift of one of the following bi-infinite
sequences:
  \begin{enumerate}
  \item $1^\Z$
  \item $B_{\gamma}. 1^{\infty}$ for some $\gamma\in \Gamma$.
  \item $1^{-\infty}.01^\infty$
\end{enumerate}
Moreover, $\C I_f$ is a Cantor set and $\w(c)\simeq\w^2+1$.
\end{theorem}

\begin{proof}

Clearly, all such full itineraries are admissible and correspond to
points in $\C I_f$.  
Now suppose that $x\in \C I_f$.  If the full itinerary of $x$ has
fewer than two $0$s, then it must satisfy either case (1) or (3).
Suppose then that it has two $0$s.  Since $00$ and $010$ do not
occur infinitely often in $K$, the full itinerary of $x$ is
$$\sh^r(...1^k01^m.01^j...)$$ for some $m\ge 2$ and $k,j\ge 1$ and
$r\ge 0$.  The only way to obtain $01^m0$ from $K$ is for this to be
of the form
$$\sh^r(...1^{k-1}B_{m-1}.B_t...)$$ where $m-1, t$ is an allowed
path up the binary tree.  By construction of $\Gamma$ we know in
fact that we have some $\gamma\in \Gamma$ which has as its terminal
segment $\gamma_n, \gamma_{n-1}, \dots \gamma_1$ where
$\gamma_n=m-1$, $\gamma_{n-1}=t$ and $\gamma_1=1$.  So we see that
in fact the full itinerary of $x$ is  $$
\sh^r(...1^{k-1}B_{\gamma_{n}}.B_{\gamma_{n-1}}\dots B_1 ...)$$
Since $n_i\to \infty$ and $B_1$ is always followed by some $1^{n_i}$
we see that the full itinerary of $x$ is
$$\sh^r(...1^{k-1}B_{\gamma_{n}}.B_{\gamma_{n-1}}\dots B_1
1^\infty)$$  It also follows that the string of $B_i$s that can
precede $B_{\gamma_n}$ in $K$ will form an infinite path through the
labeled tree.  Hence if the full itinerary of $x$ contains two $0$s
then a shift it must satisfy Case (2).

To see that $\C I_f$ is a Cantor set, we will show that it is a
compact set of countably many sub-Cantor sets. Let $C_0\subseteq \C
I_f$ be the set of points in $\C I_f$ with full itinerary of the
form
$$B_\gamma .1^\infty$$ for some $\gamma\in \Gamma$.  This set
is clearly homeomorphic with $\Gamma$ which is a Cantor set.  Hence
for each $r\in \Z$ $C_r=\sh^r(C_0)$ is also a Cantor set.  Thus $\C
I_f$ contains countably many Cantor sets. As $r\to \infty$, $C_r\to
z$ with  full itinerary $1^{-\infty}$ and as $r\to -\infty$ the
limit points of $C_r$ have full itinerary
$\sh^t(1^{-\infty}0.1^\infty)$ for some $t\in \Z$.  Hence
$$\bigcup_{r\in \Z}C_r \cup\{z_t\}_{t\in
\N}\cup\{z\}=\C I_f$$ is a Cantor set.

Next we show that $\w(c)$ is homeomorphic to $\w^2+1$.  Notice that
by the above argument and Cases (1)--(3) on the set of possible full
itineraries, we know that $x\in \w(c)$ if, and only if its itinerary
is a shift of one of the following: \begin{enumerate}
  \item $1^\infty$
  \item $B_{\gamma_n}B_{\gamma_{n-1}}\dots B_1 1^\infty$ where $\gamma_n, \gamma_{n-1}\dots\gamma _1$ is a terminal segment of some $\gamma\in \Gamma$.
  \item $1^k01^\infty$ for some $k\in \N$.
    \end{enumerate}
Suppose that $x\in \w(c)$ and its itinerary satisfies Case (2), say
$$\sh^r(B_{\gamma_n}B_{\gamma_{n-1}}\dots B_1 1^\infty)$$ where
$r<|B_{\gamma_n}|$ and its itinerary is not degenerately of Case
(3), i.e. $n>2$ and $\gamma_n\neq 2^t$  and $\gamma_n\neq 2^t+1$ for
any $t\ge 0$. Every time $B_{\gamma_n}B_{\gamma_{n-1}}$ occurs in
$K$ it is followed by $B_{\gamma_{n-2}}B_{\gamma_{n-3}}\dots B_1
1^M$ with $M$ increasing.  Therefore, there is a unique point in
$\w(c)$ which has $\sh^r(B_{\gamma_n}B_{\gamma_{n-1}}\dots B_1)$ as
an initial segment of its itinerary.  It follows then that $x$ is
isolated. Clearly, every point $y_t\in \w(c)$ with itinerary
$1^t01^\infty$ is a limit of some sequence of isolated points of
Case (2).  These all converge to a point $y\in \w(c)$ with itinerary
$1^\infty$.

\end{proof}

Notice that we do not need to use the entire binary tree in the
above construction.  By careful choice of subsets of $\Gamma$ we can
follow a modified version of our technique and force the points
$z_t\in \C I_f$ to have any limit type while the points $y_k\in
\w(c)$ have limit type 1.  As long as our subset of $\Gamma$ has an
infinite labeled path other than $(1,2, 2^2, 2^3\dots 2^n\dots)$ the
points $y_k$ will have limit type 1 and so the point $y$ will be
limit type 2.

\bibliographystyle{plain}

\bibliography{nsfgrant}

\end{document}